\documentclass[11pt,reqno,tbtags,amsfonts]{amsart}
\usepackage{amsmath}
\usepackage{amsthm}
\usepackage{amssymb}
\usepackage{amsfonts}
\usepackage{color}

\usepackage[latin1]{inputenc}
\numberwithin{equation}{section}

\newtheorem{theorem}{Theorem}[section]
\newtheorem{corollary}[theorem]{Corollary}
\newtheorem{lemma}[theorem]{Lemma}
\newtheorem{proposition}[theorem]{Proposition}
\newtheorem{claim}[theorem]{Claim}
\newtheorem{example}[theorem]{\sl Example}

\theoremstyle{definition}
\newtheorem{remark}[theorem]{Remark}

\newtheorem*{acks}{Acknowledgements}

\newcommand{\lc}{\left\lceil}

\newcommand{\rc}{\right\rceil}

\newcommand{\EE}{{\bf  E}}

\newcommand{\PP}{{\bf  P}}

\newcommand{\Fc}{{\mathcal F}}

\newcommand{\Lc}{\mathcal{L}}

\newcommand{\Leq}{{\,\stackrel{\Lc}{=}\,}}

\newcommand{\Sc}{\mathcal{S}}

\newcommand{\begp}{\begin{proposition}}
\newcommand{\enp}{\end{proposition}}
\newcommand{\begt}{\begin{theorem}}
\newcommand{\ent}{\end{theorem}}
\newcommand{\begl}{\begin{lemma}}
\newcommand{\enl}{\end{lemma}}
\newcommand{\begc}{\begin{corollary}}
\newcommand{\enc}{\end{corollary}}
\newcommand{\begcl}{\begin{claim}}
\newcommand{\encl}{\end{claim}}
\newcommand{\begr}{\begin{remark}}
\newcommand{\enr}{\end{remark}}
\newcommand{\begal}{\begin{algorithm}}
\newcommand{\enal}{\end{algorithm}}
\newcommand{\begd}{\begin{definition}}
\newcommand{\enf}{\end{definition}}
\newcommand{\begx}{\begin{example}}
\newcommand{\enx}{\end{example}}
\newcommand{\bega}{\begin{array}}
\newcommand{\ena}{\end{array}}

\newcommand{\ignore}[1]{}

\newcommand{\sfrac}[2]{{\textstyle\frac{#1}{#2}}}

\def\rompar(#1){\textup(#1\textup)}    

\newcommand{\refS}[1]{Section~\ref{#1}}
\newcommand{\refT}[1]{Theorem~\ref{#1}}

\newcommand{\refP}[1]{Proposition~\ref{#1}}

\newcommand\urladdrx[1]{{\urladdr{\def~{{\tiny$\sim$}}#1}}}

\begin{document}

\author{James Allen Fill}
\thanks{Research of the first author
supported by NSF grants DMS-0104167 and DMS-0406104 and by The Johns Hopkins
University's Acheson~J.~Duncan Fund for the Advancement of Research
in Statistics}

\address{Department of Applied Mathematics and Statistics,
The Johns Hopkins University}
\email{jimfill@jhu.edu} 
\urladdrx{http://www.ams.jhu.edu/~fill/}

\author{Mark Huber}
\thanks{Research of the second author supported by NSF CAREER grant
DMS-05-48153}
\address{Department of Mathematics \& Computer Science, Claremont
McKenna College}
\email{mhuber@cmc.edu}
\urladdrx{http://www.cmc.edu/pages/faculty/MHuber/}

\title[Perfect simulation of Vervaat perpetuities]
{Perfect simulation of Vervaat perpetuities} 



\date{revised January~8, 2010}

\begin{abstract}
We use coupling into and from the past to sample perfectly 
in a simple and provably fast fashion
from the Vervaat family 
of perpetuities.  The family includes the Dickman distribution, 
which arises both in number theory and in the analysis of the 
{\tt Quickselect} algorithm, which was the motivation for our work.
\end{abstract}

\maketitle

\section{Introduction, background, and motivation}
\label{S:intro}

\subsection{Perpetuities in general}
\label{S:general}

Define a \emph{perpetuity} to be a random variable~$Y$ such that
\begin{equation}
\label{perpdef}
Y = W_1 + W_1 W_2 + W_1 W_2 W_3 + \cdots
\end{equation}
for some sequence $W_1, W_2, W_3, \ldots$ of independent and identically 
distributed random variables distributed as~$W$.  Throughout this 
paper we assume \mbox{$W \geq 0$}\ (a.s.)\ and $\EE\,W < 1$ since these simplifying assumptions
are met by the Vervaat perpetuities, which are discussed in \refS{S:Vervaat} and are the focus
of our work.  [Some authors 
define a perpetuity as in~\eqref{perpdef} but with~$1$ added to the 
right-hand side.]
The distribution of such a random variable~$Y$ is also referred to as a 
perpetuity.  General background on perpetuities
is provided in the first paragraph 
of Devroye~\cite[Section~1]{Devroye}.  To avoid repetition, we refer 
the reader to that paper, which also cites literature about perpetuities 
in general and about approximate sampling algorithms.

Within the general framework of~\eqref{perpdef}, the following simple 
observations (with $\Leq$ denoting equality in law, or distribution)
can be made:
\begin{enumerate}
\item The random variable $Y \geq 0$ is finite almost surely; indeed, its 
expected value is finite:
$$
\EE\,Y = \frac{\EE\,W}{1 - \EE\,W} < \infty.
$$
\item The perpetuity satisfies the distributional fixed-point equation
\begin{equation}
\label{fix}
Y \Leq W (1 + Y)
\end{equation}
where, on the right, $W$ and~$Y$ are independent.  [In fact, this fixed-point 
equation characterizes the distribution~\eqref{perpdef}.]
\item If~$W$ has a density~$f_W$, then~$Y$ has a density~$f_Y$ satisfying 
the integral equation
$$
f_Y(y) = \int^{\infty}_0\!(1 + \eta)^{- 1} f_W\Big( 
  \frac{y}{1 + \eta} \Big) f_Y(\eta)\,d\eta, \quad y \geq 0.
$$
\end{enumerate}

\subsection{{\tt Quickselect}, the Dickman distribution, and the 
Vervaat family of perpetuities}
\label{S:Vervaat}

Our interest in perpetuities originated with study of the 
running time of the {\tt Quickselect} 
algorithm (also known as {\tt Find}), due to Hoare~\cite{Hoare}.  
{\tt Quickselect}$(n, m)$ is a recursive algorithm to find the item of 
rank $m \geq 1$ (say, from the bottom) in a list of $n \geq m$ distinct 
numbers (called ``keys'').  First, a ``pivot'' is chosen uniformly at 
random from among the~$n$ keys and every key is compared to it, thereby 
determining its rank (say, $j$) and separating the other keys into two 
groups.  If $j = m$, then {\tt Quickselect}$(n, m)$ returns the pivot.  
If $j > m$, then {\tt Quickselect}$(j - 1, m)$ is applied to the keys 
smaller than the pivot.  If $j < m$, then {\tt Quickselect}$(n - j, m - j)$ 
is applied to the keys larger than the pivot.

Let $C(n, m)$ denote the (random) number of key comparisons required by 
the call 
{\tt Quickselect}$(n, m)$, and write ${\bf 1}(A)$ to mean the indicator
that has value~$1$ if the Boolean expression~$A$ is 
true and~$0$ otherwise.  
Let $J_n$ denote the rank of the first pivot chosen.
Immediately from the description of the algorithm 
we find the distributional recurrence relation
\begin{equation}
\label{recur}
C(n, m) \Leq (n - 1) + {\bf 1}(J_n > m) C(J_n - 1, m) 
  + {\bf 1}(J_n < m) C^*(n - J_n, m - J_n),
\end{equation}
where, on the right, (i) for each fixed~$r$ and~$s$ the random variable 
$C(r, s)$ is distributed as the number of key comparisons required by 
{\tt Quickselect}$(r, s)$, the joint distribution of such random variables 
being irrelevant; (ii) similarly for $C^*(r, s)$; (iii) the collection 
of random variables $C(r, s)$ is independent of the collection of 
$C^*(r, s)$; and (iv) $J_n$ is uniformly distributed on $\{1, \dots, n\}$ 
and is independent of all the $C$'s and $C^*$'s.

The distribution of $C(n, m)$ is not known in closed form for general 
finite~$n$ and~$m$, so we turn to asymptotics.  For any fixed~$m$, 
formal passage to the limit as $n \to \infty$ suggests that
$$
Y(n, m) := \frac{C(n, m)}{n} - 1
$$
has a limiting distribution $\Lc(Y)$ satisfying the fixed-point equation
$$
Y \Leq U (1 + Y).
$$
This is indeed the case, as was shown by Mahmoud \emph{et al.}~\cite{MMS}.
Recalling the characterization~\eqref{fix}, we see that the law of~$Y$ 
is a perpetuity, known as the Dickman distribution.  Curiously, if~$m$ 
is chosen uniformly at random from $\{1, \dots, n\}$ before the algorithm 
is run, then the limiting distribution of $Y(n, m)$ is the convolution 
square of the Dickman distribution, as was also shown in~\cite{MMS}. 

See \cite[Section~2]{HwangTsai} concerning various important settings, 
including number theory (largest prime factors) and combinatorics 
(longest cycles in permutations), in which the Dickman distribution 
arises; also see \cite{HwangTsai} for some basic facts about this 
distribution.  The support of the distribution is $[0, \infty)$, and 
simple integral expressions are known for the characteristic 
and moment-generating functions of this log-concave 
(that is,\ strongly unimodal) 
and infinitely divisible distribution.  
In addition, the Dickman distribution has a continuous 
density~$f$ that is constant with value $e^{- \gamma}$ 
over $(0, 1]$ (here $\gamma$ is Euler's constant).  Over each 
interval $(k, k + 1]$ with~$k$ a positive integer, $f$ is the unique solution 
to the delayed differential equation
\[
f'(y) = - f(y - 1) / y.
\]
Still, no closed form for~$f$ is known.

Some years back, Luc Devroye challenged the first author to find a 
method for simulating \emph{perfectly} from the Dickman distribution 
in finite time, where one assumes that only perfect draws from the 
uniform distribution and basic arithmetic operations such as 
multiplication (with perfect precision) are possible; in particular, 
numerical integration is not allowed.  The problem 
was solved even prior to the writing of Devroye's paper~\cite{Devroye},
 but the second author pointed out extensive simplification that could be 
achieved.  The result is the present collaborative effort, which 
shows how to sample perfectly in a simple, provably fast fashion
from the Dickman distribution, and 
more generally from any member of the Vervaat~\cite{Vervaat} 
family of perpetuities handled by Devroye~\cite{Devroye}.  To 
obtain the Vervaat perpetuities, one for each value of 
$0 < \beta < \infty$, choose 
$W = U^{1 / \beta}$ in~\eqref{fix}.

\subsection{A quick review of Devroye's method}
To compare with the Markov-chain-based method employed in the 
present paper, we first review highlights of Devroye's~\cite{Devroye} 
method for perfect simulation of perpetuities.  Devroye sketches 
a general approach and carries it out successfully for the 
Vervaat family.

The underlying idea of Devroye's approach is simple acceptance-rejection:
\begin{enumerate}
\item[1.] Find explicit $h > 0$ and $1 \leq c < \infty$ so that
$$
f_Y \leq h\mbox{\ everywhere\ \ and\ \ }\int\!h = c.
$$
\item[2.] Generate~$X$ with density $c^{-1} h$.
\item[3.] Accept~$X$ with probability $f_Y(X) / h(X)$; otherwise, reject~$X$ and independently repeat steps 2--3.
\end{enumerate}
Here is how Devroye~\cite{Devroye} carries out these three steps in the case of Vervaat perpetuities.  (It is somewhat surprising how simple the first two steps turn out to be.)
\begin{enumerate}
\item[1.] One can take 
$$
\qquad \qquad h(x) := \min\left[ \beta (\beta + 1) x^{-2}, \beta x^{\beta - 1}\right]\mbox{\ \ and\ \ }c = (\beta + 1)^{(2 \beta + 1) / (\beta + 1)}.
$$
\item[2.] Let~$U_1$ and~$U_2$ be independent uniform$(0, 1)$ random variables and set
$$
X := (\beta + 1)^{1 / (\beta + 1)} U_1^{1 / \beta} U_2^{- 1}.
$$
\item[3.] Since $f_Y$ can't be computed exactly, having generated $X = x$ and an independent uniform random variable $U_3$, one must figure out how to determine whether or not $U_3 \leq f_Y(x) / h(x)$.
\end{enumerate}

Devroye's solution for step~3 is to find explicitly computable approximations $f_n(x)$ to $f_Y(x)$ and explicitly computable bounds $R_n(x)$ so that, for every~$x$,
$$
|f_n(x) - f_Y(x)| \leq R_n(x), \qquad R_n(x) \to 0\mbox{\ as $n \to \infty$}.
$$
Devroye's functions $f_n$ are quite complicated, involving (among other things)
\begin{itemize}
\item the sine integral function $\Sc(t) := \int^t_0\!\frac{\sin s}{s}\,ds$ and approximations to it computable in finite time;
\item explicit computation of the characteristic function $\phi_Y$ of~$Y$;
\item use of quadrature (trapezoidal rule, Simpson's rule) to approximate the density $f_Y$ as the inverse Fourier transform of $\phi_Y$.
\end{itemize}

Devroye proves that the running time of his algorithm is finite almost surely (for any $\beta > 0$), but cannot get finite expected running time for \emph{any} $\beta > 0$ without somewhat sophisticated improvements to his basic algorithm.  He ultimately develops an algorithm so that
$$
\EE\,T < \infty\mbox{\ for $\beta > 5 / 3$}.
$$
More careful analysis would be difficult.  Devroye makes no claim 
``that these methods are inherently practical''.  We do not mean 
to criticize; after all, as Devroye points out, his paper~\cite{Devroye} 
is useful in demonstrating that perfect simulation of perpetuities 
(in finite time) is possible.

The approach we will take is very simple conceptually, very easy 
to code, and (at least for Vervaat pereptuities with~$\beta$ not too large) very fast (provably so). 
While we have hopes that our methodology 
can be generalized to apply to any perpetuity, in this paper we develop the 
details for Vervaat perpetuities for any $\beta > 0$.

\subsection{Our approach}

Briefly, our approach to perfect simulation of a perpetuity from the 
Vervaat family is to use the Markov-chain-based perfect sampling 
algorithm known as coupling into and from the past (CIAFTP) 
(\cite{kendall1995,kendallm2000}) that produces draws exactly from
the stationary distribution of a Markov chain.  CIAFTP requires use of a 
so-called dominating chain, which for us will be simple random walk 
with negative drift on a set of the form
$\{x_0 - 1, x_0, x_0 + 1, x_0 + 2, \ldots\}$, where $x_0$ is a fixed
real number at least~$2$.  In order to handle 
the continuous state space, multigamma coupling ~\cite{murdochg1998}
will be employed.

As we shall see, all of the 
Markov chain steps involved are very easy to simulate, and the expected 
number of steps can be explicitly bounded.  For example, 
our bound is the modest constant~$15$ in the case $\beta = 1$ of the Dickman distribution,
for which the actual expected number of steps appears to be a little larger than~$6$ (consult
the end of \refS{S:time}).

The basic idea is simple.  From the discussion in~\refS{S:general} 
it is clear that the kernel~$K$ given by
\begin{equation}
\label{Kdef}
K(x, \cdot) := \Lc(W (1 + x))
\end{equation}
provides a Markov chain which, for any initial distribution, 
converges in law to the desired stationary distribution $\Lc(Y)$, 
the perpetuity.  

An \emph{update function} or \emph{transition rule} for a Markov chain on state 
space~$\Omega$ with kernel~$K$ is a function
$\phi:\Omega \times \Omega' \rightarrow \Omega$ (for some space $\Omega'$) so that for a 
random variable $W$ with a specified distribution on $\Omega'$ we have
$\Lc(\phi(x, W)) = K(x, \cdot)$ for every $x \in \Omega$.
There are of course many different choices of update function for 
any particular Markov chain.  
To employ coupling from the past (CFTP) for 
a Markov chain, an update function that is monotone suffices.  
Here, an update function $\phi$ is said to be 
\emph{monotone} if whenever $x \preceq y$ with respect to a give partial order on~$\Omega$ we have $\phi(x,w) \preceq \phi(y,w)$ for all
$w \in \Omega'$.

Consider the state space $[0,\infty)$ linearly ordered by $\leq$, and let~$W$ be distributed as 
in~\eqref{fix}.  Then
$$
\phi_{\mbox{\scriptsize natural}}(x, w) = w (1 + x)
$$
provides a natural monotone update function.  If we wish to do 
perfect simulation, it would appear at first that we are ideally 
situated to employ coupling from the past (CFTP).

However, there are two major problems:
\begin{enumerate}
\item The rule $\phi_{\mbox{\scriptsize natural}}$ is \emph{strictly} monotone, and thus no two trajectories begun at distinct states will ever coalesce.
\item The state space has no top element.
\end{enumerate}
It is well known how to overcome these difficulties:
\begin{enumerate}
\item Instead of $\phi_{\mbox{\scriptsize natural}}$, we use a 
multigamma coupler~\cite{murdochg1998}.
\item Instead of CFTP, we use CIAFTP with a dominating chain 
which provides a sort of ``stochastic process top'' to the state 
space \cite{kendall1995,kendallm2000}. 
\end{enumerate}
Our work presents a multigamma coupler for perpetuities that is 
monotone. 
Monotonicity greatly simplifies the use of CFTP~\cite{proppw1996} and CIAFTP. 
Useful monotone couplers can be difficult to find for interesting problems on continuous
state spaces; two very different 
examples of the successful use of monotone couplers on such spaces
can be found in~\cite{wilson2000b} and~\cite{huberw2009}. 

\subsection{Outline}
In \refS{S:coupler} we describe our multigamma coupler; in \refS{S:dom}, 
our dominating chain.  \refS{S:algo} puts everything together and gives a 
complete description of the algorithm, and \refS{S:time} is devoted to 
bounding the running time.  \refS{S:lit} briefly discusses approaches similar to ours
carried out by two other pairs of authors.

\section{The multigamma coupler}
\label{S:coupler}

The multigamma coupler of Murdoch and Green~\cite{murdochg1998} is an
extension of the $\gamma$ coupler described in Lindvall~\cite{lindvall1992}
that couples a single pair of random variables.  An update function can be
thought of as coupling an uncountable number of random variables
simultaneously, thus the need for ``multi"gamma coupling.  The goal
of multigamma coupling is to create an update function whose range
is but a single element with positive probability, in order to use
CIAFTP as described in Section~\ref{S:algo}.

Originally multigamma coupling was described in terms of densities; here we
consider the method from the point of view of the update function.
Suppose that the update function can be written in the form
$$
\phi(x, w) = {\bf 1}(w \in A_x)\phi_1(w) + {\bf 1}(w \notin A_x)\phi_2(x, w).
$$
Now let $A := \cap_x A_x$.  If $W \in A$, then the Markov
chain will transition from~$x$ to $\phi(x, W) = \phi_1(W)$, for all values of~$x$.  

Recall our Markov kernel~\eqref{Kdef}.  In the Vervaat case 
$W = U^{1 / \beta}$ the conditional distribution of $W (1 + x)$ 
given $W \leq 1 / (1 + x)$ is the unconditional distribution of~$W$.  

\begin{proposition} \label{P:coupler}
For any Vervaat perpetuity, the function
$$
\phi(x, w(1), w(2)) := 
  {\bf 1}\left(w(1) \leq \frac{1}{1 + x}\right) w(2) + 
{\bf 1}\left(w(1) > \frac{1}{1 + x}\right) w(1)(1 + x)
$$
is a monotone update function for the kernel~$K$ 
at~\eqref{Kdef}.  More explicitly, $\phi(x, w(1), w(2))$ is 
nondecreasing in~$x \in [0, \infty)$ for each fixed 
$w(1), w(2) \in [0, 1)$, 
and the distributions of $\phi(x, W(1), W(2))$ and $W (1 + x)$ are the same 
when $W(1), W(2)$ are two independent random variables distributed 
as $W = U^{1 / \beta}$, where~$U$ is uniformly distributed on $[0, 1)$.  
\end{proposition} 

\begin{proof}
Since the conditional distribution of $U(1 + x)^\beta$ given
$U \leq (1 + x)^{-\beta}$ is uniform on $[0,1)$, by taking
$\beta$th roots we find
\[
\Lc(W(1 + x) \mid W \leq (1 + x)^{-1}) = \Lc(W),
\]
independent of $x$, as remarked prior to the statement of the proposition.  
Therefore $\phi(x,W(1),W(2))$ has the same
distribution as $W(1 + x)$.

Also, for fixed $w(1), w(2) \in [0, 1)$ 
the function $\phi(\cdot, w(1), w(2))$ 
has the constant value $w(2) < 1$ over 
$[0, w(1)^{- 1} - 1]$ and increases linearly over 
$[w(1)^{-1} - 1, \infty)$ with value~$1$ at the left endpoint and 
slope~$w(1)$, so monotonicity is also clear.
\end{proof}

\section{A dominating chain}
\label{S:dom}

Dominating chains were introduced by Kendall~\cite{kendall1995} 
(and extended in~\cite{kendallm2000}) 
to extend the use of CFTP to chains on a partially ordered state space with a bottom element but no top element.
A \emph{dominating chain} for a Markov chain $(X_t)$ is another
Markov chain $(D_t)$ that can be coupled with $(X_t)$ 
so that $X_t \leq D_t$ for all $t$.  
In this section we give such a dominating chain
for the Vervaat perpetuity, and in the next section we describe how the
dominated chain can be used with CIAFTP.

Since our multigamma coupler~$\phi$ of \refP{P:coupler} is monotone, 
if we can find an update function $\psi(x, w(1),w(2))$ on a 
subset~$S$ of the positive integers~$x$ such that
\begin{equation}
\label{domineq}
\phi(x, w(1),w(2)) \leq \psi(x, w(1),w(2))
  \mbox{\ \ for all $x \in S$},
\end{equation}
then [with the same driving variables $W(1),W(2)$ 
as in \refP{P:coupler}] 
$\psi$ drives a dominating chain.  This is because if $x \in [0, \infty)$ 
and $y \in S$ and $x \leq y$, then
$$
\phi(x, w(1), w(2)) \leq 
  \phi(y, w(1), w(2)) \leq \psi(y, w(1), w(2)).
$$

We exhibit such a rule~$\psi$ in our next result, \refP{P:dom}.  
It is immediate from the definition of~$\psi$ that the dominating 
chain is just a simple random walk on the integers 
$\{x_0 - 1, x_0, \dots\}$ that moves left with probability $2 / 3$ 
and right with probability $1 / 3$; the walk holds at $x_0 - 1$ 
when a move to $x_0 - 2$ is proposed.  In the definition of~$\psi$, 
note that no use is made of~$w(2)$.

\begin{proposition} \label{P:dom}
Fix $0 < \beta < \infty$.  Let~$\phi$ be the multigamma coupler for 
Vervaat perpetuities described in \refP{P:coupler}.  Define
\begin{equation}
\label{EQN:x_0}
x_0 := \lc \frac{2}{1 - (2 / 3)^{1 / \beta}} \rc - 1 \geq 2
\end{equation}
and, for $x \in S := \{x_0 - 1, x_0, \dots\}$,
$$
\psi(x, w(1), w(2)) 
  := x + {\bf 1}(w(1) > (2 / 3)^{1 / \beta}) 
  -  {\bf 1}(w(1) \leq (2 / 3)^{1 / \beta},\ x \geq x_0).
$$
Then~\eqref{domineq} holds, and so~$\psi$ drives a chain that dominates the $\phi$-chain.\end{proposition}

\begin{proof}
To establish~\eqref{domineq}, suppose first that $x \geq x_0$.  
Since 
$$
\phi(x, w(1), w(2)) 
  \leq x + 1\mbox{\ \ and\ \ }
  \psi(x, w(1),w(2)) \in \{x - 1, x + 1\},
$$
we need only show that if 
$\psi(x, w(1), w(2))  
 = x - 1$ [i.e.,\ if $w(1) \leq (2 / 3)^{1 / \beta}$], then 
$\phi(x,w(1), w(2))  \leq x - 1$.  
Indeed, if $w(1) \leq 1 / (1 + x)$, then
$$
\phi(x, w(1), w(2)) ) = w(2) \leq 1 \leq x_0 - 1 \leq x - 1;
$$
and if $w(1) > 1 / (1 + x)$, then
$$
\phi(x,w(1), w(2))  
  = w(1) (1 + x) \leq (2 / 3)^{1 / \beta} (x + 1) \leq x - 1,
$$
where the last inequality holds because
\begin{equation}
\label{x0ineq}
\frac{x - 1}{x + 1} = 1 - \frac{2}{x +1} \geq 1 - \frac{2}{x_0 + 1} \geq \left( \frac{2}{3} \right)^{1 / \beta}.
\end{equation}
[This is how the value of~$x_0$ was chosen in (\ref{EQN:x_0}).]

The proof for $x = x_0 - 1$ is quite similar, so we omit most details.  In place of~\eqref{x0ineq} one needs to establish
$$
1 - \frac{1}{x_0} \geq \left( \frac{2}{3} \right)^{1 / \beta},
$$
but this follows from the last inequality in~\eqref{x0ineq} since $x_0 \geq 2 > 1$.
 
\end{proof}

\section{A perfect sampling algorithm for Vervaat perpetuities}
\label{S:algo}

For an update function $\phi(x,w)$ and random variables 
$W_{-t},\ldots,W_{- 1}$, set 
$$F_{-t}^0(x) = \phi(\phi(\ldots\phi(x,W_{-t}),\ldots,W_{-2}),W_{-1}),$$
so that if $X_{-t} = x$, then $X_0 = F_{-t}^0(x)$.
To use coupling from the past with a dominated
chain for perfect simulation, we require:

\begin{enumerate}
\item{A bottom element~$0$ for the partially ordered state space of the underlying 
chain~$X$, a dominating chain~$D$, and update functions
$\phi(x,w)$ and $\psi(x,w)$ for simulating the underlying chain 
and dominating chain forward in time.}
\item{The ability to generate a variate from the stationary distribution of
the dominating chain.  This variate is used as $D_0$, 
the state of the dominating chain at time 0.}
\item{The ability to simulate the dominating chain backwards in time, so that
$D_{-1},D_{-2},\ldots,D_{-t}$ can be generated given $D_0$.}
\item{The ability to generate i.i.d.\ random variables $W_{-t},W_{-t+1},\ldots,W_{-1}$ 
for the update function
conditioned on the values of $D_0,\ldots,D_{-t}$.}
\item{The ability (given $W_{-t},\ldots,W_{-1}$) to determine whether
$F_{-t}^0(x)$ takes on a single value for all $x \in [0,D_{-t}]$.}  (This detection of coalescence
can be conservative, but we must never claim to detect coalescence when none occurs.)
\end{enumerate}

With these pieces in place, dominated coupling from the 
past~\cite{kendallm2000} is:
\begin{enumerate}
\item[1.] Generate $D_0$ using~(ii).
\item[2.] Let $t' \leftarrow 1$, $t \leftarrow 0$.
\item[3.] Generate $D_{-t - 1},\ldots,D_{-t'}$ using (iii).
\item[4.] Generate $W_{-t'},\ldots,W_{-t-1}$ using $D_{-t'},\ldots,D_{-t}$
and~(iv).
\item[5.] If $F_{-t'}^0([0,D_{-t'}]) = \{x\}$, then output state $x$ as
the random variate and quit.
\item[6.] Else let $t \leftarrow t'$, $t' \leftarrow t' + 1$, and go to step 3.
\end{enumerate}

The update to $t'$ in step 6 can be done in several ways.  For instance,
doubling time rather than incrementing it additively is often done.  For our
chain step 5 only requires checking to see 
whether $W_{-t'} \leq 1 / (1 + D_{-t'})$, 
and so requires constant time for failure and time $t'$ for success.  Thus step~6
as written is efficient.

Because the dominating chain starts at time 0 and is simulated into the past,
and then the underlying chain is (conditionally) simulated forward, 
Wilson~\cite{wilson2000b} referred to this algorithm as coupling into and
from the past (CIAFTP).

Now we develop each of the requirements (i)--(v) for Vervaat perpetuities.  
Requirement~(i) was dealt with in Section~\ref{S:coupler}, where it was 
noted that the dominating chain is a simple
random walk with probability $2/3$ of moving down one unit towards~$0$
and probability $1/3$ of moving up one unit.
The chain has a partially absorbing barrier at $x_0 - 1$, 
so that if the chain is at state $x_0 - 1$ and tries
to move down it stays in the same position.

The stationary distribution of the dominating random walk
can be explicitly calculated and is a shifted geometric random
distribution; requirement~(ii) can be satisfied by letting $D_0$
be $x_0 - 2 + G$ where $G \in \{1, 2, \dots\}$ has the Geometric distribution
with success probability $1 / 2$.  [Recall that we can generate~$G$ from a
Unif$(0, 1)$ random variable~$U$ using $G = \lceil -(\ln U) / (\ln 2) \rceil$.]
Simulating the $D$-chain backwards in time [requirement~(iii)] 
is also easy since~$D$ is a birth-and-death chain and so
is reversible.

Now consider requirement~(iv).  Suppose that a one-step transition $D_{-t+1}$ $ = D_{-t } + 1$ (say) forward in time is observed, 
and consider the conditional distribution of the driving variable $W_{-t}$ that would produce this [via $D_{-t + 1} = \psi(D_{-t}, W_{-t})$]. 
What we observe is precisely that the random variable $W_{-t}(1) = U_{-t}^{1/\beta}$ 
fed to the update function~$\psi$ must have satisfied
$W_{-t}(1) > (2/3)^{1/\beta}$, i.e.,\ $U_{-t} > 2/3$.  
Hence we simply generate $W_{-t}(1)$ from the distribution of $U_{-t}^{1/\beta}$ 
conditionally given $U_{-t} > 2 / 3$.  Similarly, if 
$D_{-t + 1} = D_{-t} - 1$ or $D_{-t + 1} = D_{-t} = x_0 - 1$, then we generate $W_{-t}(1)$ conditioned
on $W_{-t}(1) \leq (2/3)^{1/\beta}$, i.e.,\ on $U_{-t} \leq 2/3$.  The random variable $W_{-t}(2)$ is always generated independently as the $1 / \beta$ power of a Unif$(0, 1)$ random variable.

Finally, requirement~(v) is achieved by using the multigamma coupler from the last section;
indeed, if ever $W_{-t}(1) \leq 1/(D_{-t} + 1)$,
then the underlying chain coalesces to a single state at time $-t + 1$, and hence also at time~$0$. 
\smallskip 

Thus our algorithm is as follows:
\newpage
\begin{enumerate}
\item[1.] Generate $G \sim \mbox{Geom}(1 / 2)$ and let 
$D_0 \leftarrow x_0 - 2 + G$.  Set $t \leftarrow 1$.
\item[2.] Having generated $D_{- (t - 1)}$, generate 
$D_{-t}$ by moving up from $D_{- (t - 1)}$ with 
probability $1 / 3$ and ``down'' with probability $2 / 3$.  
Here, and in step~4 below, a move up adds~$1$ and a move 
``down'' subtracts~$1$ unless $D_{- (t - 1)} = x_0 - 1$, 
in which case $D_{- t} = x_0 - 1$, too.
\item[3.] Read the transition forward from 
$D_{- t}$ to $D_{-t + 1}$.
\item[4.] If the forward move is up, impute 
$U_{- t} \sim \mbox{Unif}(2 / 3, 1)$; 
if the forward move is ``down'', impute $U_{- t} \sim \mbox{Unif}(0, 2 / 3]$.  In either case, 
set $W_{- t}(1) = U^{1 / \beta}_{ - t}$.
\item[5.] If $W_{- t}(1) \leq 1 / (D_{- t} + 1)$, 
then generate $W_{- t}(2)$ independently and set $X_{- t + 1} = W_{- t}(2)$;
otherwise increment~$t$ by~$1$ and return to step~2.
\item[6.] Using the imputed driving variables $W_{-s}(1)$ ($t \geq s \geq 1$), and generating independent variables $W_{-s}(2)$ as necessary, grow the $X$-trajectory forward to time~$0$:
$$
X_{- s + 1} = \phi(X_{- s}, W_{-s}(1), W_{-s}(2)).
$$
Here~$\phi$ is the multigamma coupler described in \refP{P:coupler}.
\item[7.] Output $X_0$.
\end{enumerate}

\section{A bound on the running time}
\label{S:time}

\begin{theorem}  
\label{T:bound}
Let~$T$ be the number of steps taken backwards in time.
(This is also the number of times steps~3, 4, and 5 of the algorithm at the end of 
Section~\ref{S:algo} are executed.)  Define 
$x_0 = \lceil 2 / (1 - (2/3)^{1/\beta}) \rceil - 1$ as in
Proposition~\ref{P:dom}.  Then for any $0 < \beta < \infty$ we have
\begin{equation}
\label{bounds}
x_0^{\beta} \leq \EE\,T \leq 2(x_0 + 1)^\beta + 3
\end{equation}
and 
hence $\EE\,T = e^{\beta \ln \beta + \Theta(\beta)}$ as $\beta \to \infty$;
and
\begin{equation}
\label{eff}
\EE\,T \to 1\mbox{\rm \ as $\beta \to 0$}.
\end{equation}
\end{theorem}

\begin{proof}
The lower bound in~\eqref{bounds} is easy:  Since $D_t \geq x_0 - 1$, the expectation of the number of steps~$t$ needed to get $(U_{-t})^{1  /\beta} \leq 1/(D_{-t} + 1)$ is at least 
$x_0^\beta$.  We proceed to derive the upper bound.

Consider a potential function $\Phi = (\Phi_t)_{t \geq 0}$ such that $\Phi_t$
depends on $D_{-t},\ldots,D_{0}$ 
and $W_{-t},\ldots,W_{-1}$ in the following way:
\begin{equation}
\Phi_t = \left\{ \begin{array}{ll}
 0   & \text{if } t \geq T \\
D_{-t} - (x_0 - 1) + \frac{2}{3}( x_0 + 1)^\beta & \text{otherwise}.
\end{array} \right.
\end{equation}
We will show that the
process $(\Phi_{t \wedge T} + (1/3)(t \wedge T))$ 
is a supermartingale
and then apply the optional sampling theorem to derive the upper bound.  
Let $\Fc_t$
be the $\sigma$-algebra generated by $D_{-t},\ldots,D_0$ and 
$W_{-t},\ldots,W_{-1}$, so that $T$ is a stopping time with respect to 
the filtration $(\Fc_t)$.  Suppose that $T > t$.  In this case
there are two components to 
$\Phi_t - \Phi_{t-1}$.  The first component comes from the change from 
$D_{-t}$ to $D_{-t+1}$.  The second component comes from the
possibility of coalescence (which gives $T = t$) 
due to the choice of $W_{-t}(1)$.
The expected change in $\Phi$ is the sum of the expected changes 
from these two sources.

For the change in the $D$-chain we observe
\begin{eqnarray}
\lefteqn{\hspace{-.2in}\EE[D_{-t} - D_{-t+1} |\Fc_{t-1}]} \nonumber \\ 
 & = & (\sfrac{1}{3} - \sfrac{2}{3}) {\bf 1}(D_{-t+1} > x_0 - 1) + \sfrac{1}{3} {\bf 1}(D_{-t+1} = x_0 - 1). \label{E:change_in_D}
\end{eqnarray}

The expected change in~$\Phi$ due to coalescence can be bounded 
above by considering coalescence only when $D_{-t+1} = x_0 - 1$, 
in which case we have $D_{-t} \in \{x_0 - 1,x_0\}$.  
But then
coalescence occurs with probability at 
least $[1/(1 + x_0)]^{\beta}$, 
and if coalescence occurs, $\Phi$ drops from
$(2/3)(x_0 + 1)^\beta$ down to~$0$.
Therefore the expected change in~$\Phi$ from
coalescence when $D_{-t+1} = x_0 - 1$ is at most
$-(2/3)(x_0 + 1)^\beta / (x_0 + 1)^\beta = -2/3.$ 
Combining with~\eqref{E:change_in_D} yields
\begin{eqnarray*}
\label{E:change_in_D2}
\lefteqn{\EE[\Phi_t - \Phi_{t-1} | \Fc_{t-1}]} \\ 
& \leq & (\sfrac{1}{3} - \sfrac{2}{3}) {\bf 1}(D_{-t+1} > x_0 - 1) + (\sfrac{1}{3} - \sfrac{2}{3}) 
{\bf 1}(D_{-t+1} = x_0 - 1) = -\sfrac{1}{3}.
\end{eqnarray*}
So whenever $T > t$, the potential~$\Phi$ decreases by at least $1/3$ on average at each step,
and hence $(\Phi_{t \wedge T} + (1/3) (t \wedge T))$ 
is a supermartingale.  Since it is also nonnegative
the optional sampling theorem can be applied 
(see for example~\cite{durrett2005}, p.~271)
to yield
$$
\EE\,\Phi_T + \sfrac{1}{3} \EE\,T \leq \EE\,\Phi_0.
$$
Since $\Phi_T = 0$, it remains to bound $\EE\,\Phi_0$. 
Recall that $D_0 - (x_0 - 2)$ is a geometric random variable with parameter $1/2$ and
thus has mean~$2$.  Hence $\EE\,\Phi_0 = 1 + (2/3)(x_0 + 1)^\beta$, giving the desired upper bound on $\EE\,T$.

We now prove~\eqref{eff}, using notation such as $T(\beta)$ to indicate explicitly the dependence of various quantities on the value of the parameter~$\beta$.  However, notice that $x_0(\beta) = 2$ for all $0 < \beta \leq \beta_0 := \ln(3/2) / \ln 3$ and hence that the same dominating chain~$D$ may be used for all such values of~$\beta$.  For $0 < \beta \leq \beta_0$ define 
\begin{equation}
\label{Ytbeta}
Y_t(\beta) := \prod_{s = 1}^t \left[ 1 - (D_{-s} + 1)^{-\beta} \right].
\end{equation}
Conditionally given the entire $D$-process, since 
$1 - (D_{-s} + 1)^{-\beta}$ is the probability that coalescence does not
occur at the $s$th step backwards in time, 
$Y_t(\beta)$ equals the probability that 
coalescence does not occur on any of the first $t$ steps backwards in time.
Thus
$$
\PP(T(\beta) > t) = \EE\,Y_t(\beta), \quad t = 1, 2, \dots,
$$
and hence
\begin{equation}
\label{series}
\EE\,T(\beta) = 1 + \sum_{t = 1}^{\infty} \EE\,Y_t(\beta).
\end{equation}
We now need only apply the dominated convergence theorem, observing that $Y_t(\beta) \geq 0$ is increasing in $0 < \beta \leq \beta_0$ and that $1 + \sum_{t = 1}^{\infty} \EE\,Y_t(\beta_0) = \EE\,T(\beta_0) \leq 2 \cdot 3^{\beta_0} + 3 = 6 < \infty$ by~\eqref{bounds}.
\end{proof}

For any fixed value of~$\beta$, it is possible to find $\EE\,T$ to any desired precision.  
Let us sketch how this is done for the Dickman distribution, where $\beta = 1$.  First, we note for simplicity
that~$T$ in \refT{T:bound} has the same distribution as the time it takes for the dominating random walk
$D = (D_t)_{t \geq 0}$, run \emph{forward} from the stationary distribution at time~$0$, to stop, where ``stopping'' is determined as follows.
Having generated the walk~$D$ through time~$t$ and not stopped yet, let $U_t$ be an independent uniform$(0, 1)$ random
variable.  If $U_t > 2 / 3$, let~$D$ move up (and stopping is not possible); 
if $U_t \leq 2 / 3$, let~$D$ move ``down'' and stop (at time $t + 1$) if $U_t \leq 1 / (D_t + 1)$.
Adjoin to the dominating walk's state space $\{x_0 - 1, x_0, x_0 + 1, \ldots\}$ an absorbing state corresponding to stopping.
Then the question becomes: What is the expected time to absorption?  
We need only compute the expected time to absorption from each possible deterministic initial state
and then average with respect to the shifted-geometric stationary distribution for~$D$.
For finite-state chains, such calculations are straightforward (see for 
example~\cite{lawler1995}, pp.\ 24--25).  For infinite-state absorbing chains such as the one we have created here, 
some truncation is necessary to achieve lower and upper bounds.  This can be done using ideas similar to those
used in the proof of \refT{T:bound}; we omit further details. 

The upshot is that is takes an average of $6.07912690331468130722\dots$ steps to reach coalescence.  This is much closer 
to the lower bound given by \refT{T:bound} of~$5$ than to the upper bound of~$15$ .  
Our exact calculations confirm results we obtained from simulations.
Ten million trials (which took only a few minutes to run and tabulate using {\tt Mathematica} 
code that was effortless to write) gave an estimate of $6.07787$ with a standard error of $0.00184$.  
The algorithm took only a single Markov chain step about $17.4\%$ of the time; more than four steps 
about $47.6\%$ of the time; more than eight steps about $23.4\%$ of the time; and more than twenty-seven 
steps about $1.0\%$ of the time.  In the ten million trials, the largest number of steps needed was~$112$. 

Similar simulations for small values of~$\beta$ led us to conjecture, for some constant~$c$ near~$1$,  the refinement
\begin{equation}
\label{expand}
\EE\,T(\beta) = 1 + (1 + o(1))\,c\,\beta\mbox{\ as $\beta \to 0$}
\end{equation}
of~\eqref{eff}.
The expansion~\eqref{expand} (which, while not surprising, demonstrates extremely efficient perfect simulation 
of Vervaat perpetuities when~$\beta$ is  small) does in fact hold with
\begin{equation}
\label{cdef}
c = \sum_{i = 1}^{\infty} 2^{- i} \ln(i + 1) \approx 1.016.
\end{equation}
We will prove next that the right side of~\eqref{expand} provides a lower bound on $\EE\,T(\beta)$.  Our proof that it also
provides an upper bound uses the same absorbing-state approach we used to compute $\EE\,T(\beta)$ numerically in the case
$\beta = 1$ and is rather technical, so we omit it here. 

Using~\eqref{series} and~\eqref{Ytbeta} together with the inequality $e^{-x} \geq 1 - x$ for $x \geq 0$ we find
\begin{equation}
\label{DCTready}
\frac{\EE\,T(\beta) - 1}{\beta} > \EE \left[ \frac{1 - (D_{-1} + 1)^{- \beta}}{\beta} \right].
\end{equation}
Application of the dominated convergence theorem to the right side of ~\eqref{DCTready} gives the lower bound in~\eqref{expand}, where $c = \EE\,\ln(D_{-1} + 1)$ is given by the series~\eqref{cdef}.

\section{Related work}
\label{S:lit}
An unpublished early draft of this paper has been in existence for a number of years.  Perfect simulation of Vervaat perpetuities has been treated independently by Kendall and Th\"{o}nnes~\cite{KT}; their approach is quite similar (but not identical) to ours.  One main difference is their use of the multi-shift coupler of Wilson~\cite{wilson2000b} rather than the multigamma coupler.  The simulation with (in our notation) 
$\beta = 10$ reported in their Section~3.5 suggests that, unlike ours, their algorithm is reasonably fast even when~$\beta$ is large; it would be very interesting to have a bound on the expected running time of their algorithm to that effect.

The early draft of this paper considered only the Dickman distribution and used an integer-valued dominating chain~$D$ with a stationary distribution that is Poisson shifted to the right by one unit.  Luc Devroye, who attended a lecture on the topic by the first author of this paper in~2004, and his student Omar Fawzi have very recently improved this approach to such an extent that the expected number of Markov chain steps they require to simulate from the Dickman distribution is proven to be less than $2.32$; see~\cite{DF}.  It is not entirely clear that their algorithm is actually faster than ours, despite the larger average number $6.08$ of steps for ours (recall the penultimate paragraph of~\refS{S:time}), since it is possible that one step backwards in time using their equation~(1) takes considerably longer than our simple up/down choice (recall step~2 in the algorithm at the end of \refS{S:algo}).

\begin{acks}
We thank the anonymous referee for several helpful comments, most notably the suggestion that we try to compute the exact value of $\EE\,T$
in the case $\beta = 1$ of the Dickman distribution.
Research for the second author was carried out in part while affiliated with the Department of Mathematics, Duke University.
\end{acks}

\newcommand\AAP{\emph{Adv.\ Appl.\ Probab.} }
\newcommand\JAP{\emph{J.\ Appl.\ Probab.} }
\newcommand\JAMS{\emph{J. \AMS} }
\newcommand\MAMS{\emph{Memoirs \AMS} }
\newcommand\PAMS{\emph{Proc.\ \AMS} }
\newcommand\TAMS{\emph{Trans.\ \AMS} }
\newcommand\AnnMS{\emph{Ann.\ Math.\ Statist.} }
\newcommand\AnnPr{\emph{Ann.\ Probab.} }
\newcommand\CPC{\emph{Combin.\ Probab.\ Comput.} }
\newcommand\JMAA{\emph{J.\ Math.\ Anal.\ Appl.} }
\newcommand\RSA{\emph{Random Struct.\ Alg.} }
\newcommand\ZW{\emph{Z.\ Wahrsch.\ Verw.\ Gebiete} }
\newcommand\DMTCS{\jour{Discr.\ Math.\ Theor.\ Comput.\ Sci.} }

\newcommand\AMS{Amer.\ Math.\ Soc.}
\newcommand\Springer{Springer}
\newcommand\Wiley{Wiley}

\newcommand\vol{\textbf}
\newcommand\jour{\emph}
\newcommand\book{\emph}
\newcommand\inbook{\emph}
\def\no#1#2,{\unskip#2, no.~#1,} 

\newcommand\webcite[1]{\hfil\penalty0\texttt{\def~{\~{}}#1}\hfill\hfill}
\newcommand\webcitesvante{\webcite{http://www.math.uu.se/\~{}svante/papers/}}
\newcommand\arxiv[1]{\webcite{arXiv:#1}}

\def\nobibitem#1\par{}


\end{document}